\documentclass[11pt]{amsart}
\usepackage{amsmath,amssymb,amscd,amsthm,amsxtra}
\usepackage{latexsym}

\usepackage[dvips]{graphicx}        

\newtheorem{theorem}{Theorem}[section]
\newtheorem{lemma}{Lemma}[section]
\newtheorem{proposition}{Proposition}[section]
\newtheorem{remark}{Remark}[section]

\newtheorem{definition}{Definition}[section]
\newtheorem{corollary}{Corollary}[section]

\newcommand{\e}{\varepsilon}

\newcommand{\p}{\partial}

\newcommand{\wh}{\widehat}

\newcommand{\Id}{{\bf 1}}

\begin{document}
\title[Fifth order KdV type equations]
{Local well-posedness of periodic fifth order KdV type equations}

\author{Yi Hu}
\address{
Yi Hu\\
Department of Mathematics\\
University of Illinois at Urbana-Champaign\\
Urbana, IL, 61801, USA}

\email{yihu1@illinois.edu}

\author{Xiaochun Li}

\address{
Xiaochun Li\\
Department of Mathematics\\
University of Illinois at Urbana-Champaign\\
Urbana, IL, 61801, USA}

\email{xcli@math.uiuc.edu}

\thanks{ This work was partially supported by an NSF grant DMS-0801154}

\begin{abstract} In this paper, the local well-posedness of periodic fifth order 
dispersive equation with nonlinear term $P_1(u)\p_xu + P_2(u)\p_x u\p_xu $.
Here $P_1(u)$ and $P_2(u)$ are polynomials of $u$.  
We also get some new Strichartz estimates.    
\end{abstract} 

\maketitle

\section{Introduction}
\setcounter{equation}0

In this paper, we consider the following Cauchy problem on the fifth order dispersive equations: 
\begin{equation}\label{gKdV3}
\begin{cases}
\p_t u +\p_x^5 u+ P_1(u)\p_x u + P_2(u)\p_xu \p_x u=0\\
u(x,0)=\phi(x),\qquad x\in\mathbb{T},\ t\in\mathbb{R}\,.
\end{cases}
\end{equation}
where $P_1$ and $P_2$ are polynomials.  

\begin{theorem}\label{thm-gKdV3}
The Cauchy problem (\ref{gKdV3}) 
is locally well-posed provided that the initial data $\phi\in H^s$ for $s > 1$. 
\end{theorem}

The index $s=1$ is sharp for (\ref{gKdV3}) to be well-posed (See Section \ref{cE}).
If the nonlinear term $ P_2(u)\p_x u \p_xu $ in (\ref{gKdV3}) is removed, then we may get a better regularity 
condition on $s$. More precisely, we have

\begin{theorem}\label{LWP-KdV}
Let $P_1$ be a polynomial of degree $k\geq 2$. Then the Cauchy problem 
\begin{equation}\label{gKdV}
\begin{cases}
\p_t u+ \p_x^5 u+ P_1(u)u_x=0\\
u(x,0)=\phi(x),\qquad x\in\mathbb{T},\ t\in\mathbb{R}\,.
\end{cases}
\end{equation}
 is locally well-posed if the initial data $\phi\in H^s$ for $s>1/2$.  
\end{theorem}

Even for $P_1=0$ in (\ref{gKdV3}), the sharp regularity condition is still $s\geq 1$. 
In this case, the following well-posedness can be established.  
\begin{theorem}\label{thm-gKdV2}
The Cauchy problem 
\begin{equation}\label{gKdV2}
\begin{cases}
\p_t u +\p_x^5 u+ P_2(u)\p_xu \p_x u=0\\
u(x,0)=\phi(x),\qquad x\in\mathbb{T},\ t\in\mathbb{R}\,.
\end{cases}
\end{equation}
is locally well-posed provided that the initial data $\phi\in H^s$ for $s> 1$.  
\end{theorem}

\begin{remark}
If $P_2$ is a polynomial of degree $0$ or $1$, then Theorem \ref{thm-gKdV2} holds 
for the endpoint $s=1$. 
\end{remark}

If $P_1$ is a polynomial of degree $1$, the local well-posedness of (\ref{gKdV}) for $s > 0$ was proved by Bourgain in \cite{B3}. Moreover, in the same paper, Bourgain proved
(\ref{gKdV3}) is locally well-posed if $s$ is sufficiently large. 
Only the lower order derivative of $u$ is allowed in the nonlinear term of (\ref{gKdV3}), because the ill-posedness 
of 
\begin{equation}\label{gKdV4}
\begin{cases}
\p_t u +\p_x^5 u+ u^2\p^2_xu =0\\
u(x,0)=\phi(x),\qquad x\in\mathbb{T},\ t\in\mathbb{R}\,.
\end{cases}
\end{equation}
even for smooth initial data $\phi$ was observed by Bourgain in \cite{B3}.
Theorem \ref{thm-gKdV3} is still true even if polynomials $P_1$ and $P_2$ are replaced by sufficiently smooth functions.  One may utilize the ideas in \cite{HL2} to obtain this result.  
For technical simplicity, in this paper, we do not provide the details on the general smooth nonlinear terms.   
In what follows, we only need to prove Theorem \ref{LWP-KdV} and Theorem \ref{thm-gKdV2}, since Theorem \ref{thm-gKdV3} can be done similarly.  The higher order dispersive equations associated with smooth nonlinear terms will 
be studied in our next paper.\\

As we did in \cite{HL2}, in order to prove Theorem \ref{LWP-KdV} and Theorem \ref{thm-gKdV2}, we need 
 to build up some Stricharz inequalities. 
Let $K_{d, p,N}$ be  the best constant  satisfying 
\begin{equation}\label{Stri}
\left\|\sum_{n=-N}^N a_n e^{2\pi i t n^d+ 2\pi i x n }\right\|_{L^p_{x,t}(\mathbb T\times \mathbb T)}\leq K_{d, p, N} \left(
\sum_{n=-N}^N |a_n|^2\right)^{\frac{1}{2}}\,.
\end{equation} 
This is the periodic Strichartz inequality associated to higher order dispersive equations.
First $L^6$ estimate can be established. 

\begin{theorem}\label{thm0}
Let $K_{d, p, N}$ be defined as in (\ref{Stri}). If $d$ is odd, then for any $\e>0$, there exists a
constant $C$ independent of $N$ such that
\begin{equation}\label{L6res-est}
 K_{d, 6, N}\leq C N^\e\,. 
\end{equation}
\end{theorem}

Second, for large $p$ we have sharp estimates (up to a factor of $N^\e$).

\begin{theorem}\label{thm1}
Let $K_{d, p, N}$ be defined as in (\ref{Stri}). If $p\geq p_0$, then for any $\e>0$, there exists a
constant $C$ independent of $N$ such that
\begin{equation}\label{res-est}
K_{d, p, N} \leq C N^{\frac12-\frac{(d+1)}{p}+ \e}\,. 
\end{equation}
Here $p_0$ is given by
\begin{equation}
 p_0 = 
\begin{cases}
 (d-2)2^d +6 \quad & \text{if }\,\,\, d\,\, {\text{is odd}} \\
  (d-1)2^d+4  \quad & \text{if }\,\,\, d\,\, {\text{is even}}
\end{cases}
\end{equation}
\end{theorem} 

In terms of the language of discrete restriction, Theorem \ref{thm1} is equivalent to
\begin{equation}\label{res-1}
\sum_{n=-N}^{N} \left| \wh f(n, n^d)\right|^2\leq  CN^{1-\frac{2(d+1)}{p}+ \e}\|f\|_{p'}^2 \,\,\,{\rm for}\,\,\, p\geq p_0\,.
\end{equation}
Here $f$ is a periodic function on $\mathbb T^2 $, $\wh f$ is Fourier transform of $f$ on $\mathbb T^2$,
$ d\geq 3$ is an integer,  $p\geq 2$ and $p'=p/(p-1)$. 
The $d=2$ case was investigated by Bourgain \cite{B1}. 
It was proved by Bourgain in \cite{B2} that $K_{d, 4, N}\leq C$ and $K_{3,6,N}\leq N^\e$.
The results for large $p$ and $d=3$ were established in \cite{HL2}. The main ideas utilized in 
this paper come from \cite{H-L} and \cite{HL2}.


\section{Two Counter Examples}\label{cE}

In this section, we give two examples showing that the indices $\frac{1}{2}$ in Theorem \ref{LWP-KdV} and $1$ in Theorem \ref{thm-gKdV2} are sharp. More precisely, one has analytic ill-posedness in (\ref{gKdV}) if $s<\frac{1}{2}$ and in (\ref{gKdV2}) if $s<1$. The two examples provided in this section are  simple modifications of those in \cite{B3} and \cite{Tao}.\\

First consider (\ref{gKdV}) and take $P_1(u)=u^2$. Define the iterates $u^{(0)}$ and $u^{(1)}$ by
\begin{eqnarray}
& \p_t u^{(0)} + \p_x^5 u^{(0)} = 0, \quad u^{(0)}(x,0) = \phi(x),  \label{equ-u0}\\
&  \p_t u^{(1)} + \p_x^5 u^{(1)} + \left(u^{(0)}\right)^2\p_x u^{(0)}= 0, \quad u^{(1)}(x,0) = \phi(x).  \label{equ-u1}
\end{eqnarray}
In order for (\ref{gKdV}) to be locally wellposed, we must have
\begin{equation}\label{bound}
\sup_{0<t<\delta}\left\|u^{(1)}\right\|_{H_x^s}<\infty
\end{equation}
for some positive small $\delta$.

Let $N$ be a (large) positive integer and let
\begin{equation}\label{phi}
\phi(x)=\frac{\varepsilon}{N^{s}}e^{iNx} + \frac{\varepsilon}{N^{s}}e^{-iNx} 
\end{equation}
be a specific initial value. Obviously $\phi\in H^s$. Then $u^{(0)}$ in (\ref{equ-u0}) equals 
\begin{equation}\label{u0}
 \frac{\varepsilon}{N^{s}}e^{iNx}e^{-iN^5t} + \frac{\varepsilon}{N^{s}}e^{-iNx}e^{iN^5t}.
\end{equation}
Thus in (\ref{equ-u1}), the nonlinear term $\left(u^{(0)}\right)^2\p_x u^{(0)}$ can be expressed as 
\begin{equation}\label{nonlinear-u0}
i\varepsilon^3 N^{1-3s} \left( e^{iNx}e^{-iN^5t} - e^{-iNx}e^{iN^5t} + e^{i3Nx}e^{-i3N^5t} - e^{-i3Nx}e^{i3N^5t} \right)\,.
\end{equation}
A simple calculation, via a use of (\ref{nonlinear-u0}) and Duhamel's formula,  
allows us to represent $u^{(1)}$ as
\begin{equation}\label{u1}
 u^{(1)}(x,t)= \left( \varepsilon N^{-s} - i\varepsilon^3 N^{1-3s}t \right) e^{-iN^5t} e^{iNx} + \cdots\,.
\end{equation}
The remaining term "$\cdots$" in (\ref{u1}) is of the form 
\begin{equation}\label{form1}
 \sum_{\ell\neq 1} C_\ell e^{i\ell Nx} f_\ell(t)\, 
\end{equation}
where $f_\ell$'s are functions of the time variable $t$ only. Henceforth
using the definition of $H^s$ norm, we obtain
\begin{equation}
\|u^{(1)}\|_{H^s_x}\geq  C\e^3 N^{1-2s}t\,. 
\end{equation}
This shows $s$ must be at least $\frac{1}{2}$, since otherwise $\left\|u^{(1)}\right\|_{H^s_x}$ would blow up as $N$ goes to infinity, which contradicts (\ref{bound}).
This example can be simply modified for the case when $ P_1(u)=u^{2k}$. Hence $s=1/2$ is the best 
regularity condition for (\ref{gKdV}) to be well-posed. 
\\

Next we consider (\ref{gKdV2}) and take $P_2(u)=u$. Define the iterates $u^{(0)}$ and $u^{(1)}$ by
\begin{eqnarray}
& \p_t u^{(0)} + \p_x^5 u^{(0)} = 0, \quad u^{(0)}(x,0) = \phi(x),  \label{equ1-u0}\\
&  \p_t u^{(1)} + \p_x^5 u^{(1)} + u^{(0)}\left(\p_x u^{(0)}\right)^2= 0, \quad u^{(1)}(x,0) = \phi(x).  \label{equ1-u1}
\end{eqnarray}
Similarly, local well-posedness implies (\ref{bound}). Take the same initial value as in (\ref{phi}), so for (\ref{equ1-u0}) we get the same $u^{(0)}$ as in (\ref{u0}). Thus in (\ref{equ1-u1}), the nonlinear term $u^{(0)}\left(\p_x u^{(0)}\right)^2$ can be expressed as
\begin{equation}\label{nonlinear1-u0}
u^{(0)}\left( \p_x u^{(0)}\right)^2 = \varepsilon^3 N^{2-3s} \left( e^{iNx}e^{-iN^5t} + e^{-iNx}e^{iN^5t} - e^{i3Nx}e^{-i3N^5t} - e^{-i3Nx}e^{i3N^5t} \right).
\end{equation}
Again by (\ref{nonlinear1-u0}) and Duhamel's formula, one may represent $u^{(1)}$ as
\begin{equation}\label{u1-1}
u^{(1)}(x,t)=\left( \varepsilon N^{-s} - i\varepsilon^3 N^{2-3s}t \right) e^{-iN^5t} e^{iNx} + \cdots.
\end{equation}
Here "$\cdots$" is of the form (\ref{form1}). From (\ref{u1-1}), we get immediately
\begin{equation}
\|u^{(1)}\|_{H^s_x}\geq C\e^3 N^{2-2s}t\,, 
\end{equation}
which implies $s\geq 1$, since otherwise $\left\|u^{(1)}\right\|_{H^s_x}$ would blow up as $N$ goes to infinity, which contradicts (\ref{bound}). This example can be also generalized to the general polynomial case. \\

\section{Proof of Theorem \ref{thm0}}
\setcounter{equation}0

Via a direct calculation, we reduce the problem to count the number of integral solutions of
 \begin{equation}\label{sys0}
\begin{cases}
 n_1+ n_2 +n_3   =  A \\
 n_1^d+ n_2^d + n_3^d = B  \,.
\end{cases}
\end{equation}
Here $A, B$ are fixed constants such that $ |A|\leq 3N$ and $ |B|\leq 3N^d$. 
Write $n_3=A-n_1-n_2$ in the second equation and we then obtain
\begin{equation}\label{e1}
n_1^d+n_2^d+(A-n_1-n_2)^d = B\,.
\end{equation} 
Applying the binomial theorem, we get
\begin{equation}\label{e2}
 -\sum_{k=1}^{\frac{d-1}{2}} C(d,k)n_1^kn_2^k(n_1^{d-2k}+n_2^{d-2k})
 + \sum_{k=1}^{d-1}C(d, k) A^{d-k}(-1)^k (n_1+n_2)^k = B-A^d\,.
\end{equation} 
Here $C(d, k)$ stands for the binomial coefficient. 
Since $d-2k$ is odd, $n_1+n_2$ is a factor of the left hand side of (\ref{e2}). Henthforth
we have 
\begin{equation}\label{div1}
 (n_1+n_2)|(B-A^d)\,.
\end{equation}  
By symmetry, we get immediately that 
$n_1+n_2 $, $n_2+n_3 $ and $n_1+n_3$ are divisors of $B-A^d$. Therefore 
Theorem \ref{thm0} follows since there are at most $N^\e$ divisors of 
$B-A^d$. This completes the proof of Theorem {\ref{thm0}}.\\

In the end of this section, let us state a useful theorem on $L^4$ estimate, proved by Bourgain in \cite{B2}. A consequence of Theorem \ref{B4est} is, in terms of $X_{s, b}$ defined as in 
Definition \ref{defofXsb},  $X_{0, \frac{d+1}{4d}}\subset L^4_{loc}$.

\begin{theorem}\label{B4est}
For any function $f$ on $ \mathbb T^2$, 
\begin{equation}\label{B4est1}
 \|f\|_4\leq C \left (\sum_{m, n\in \mathbb Z} (1+|n-m^d|)^{\frac{d+1}{2d}}|\wh f(m,n)|^2  \right)^{1/2}\,.
\end{equation}
\end{theorem}

\section{Proof of Theorem \ref{thm1}}
\setcounter{equation}0

The argument in this section is a modification of those in \cite{H-L} and \cite{HL2}.
For the sake of self-containedness, we present all details here. 
To prove Theorem {\ref{thm1}}, we need to introduce a level set.
Let $F_N $ be a periodic function on $\mathbb T^2$ given by
\begin{equation}\label{defofF0}
 F_N(x, t) = \sum_{n=-N}^N a_{n} e^{2\pi i nx} e^{2\pi i n^d t}\,,
\end{equation}
where $\{a_n\}$ is a sequence with $\sum_n |a_n|^2 =1$ and $(x, t)\in \mathbb T^2$. For any 
$ \lambda>0 $, set a level set $E_\lambda$ to be
\begin{equation}\label{defofElam}
E_\lambda = \left\{ (x, t)\in \mathbb T^{2}: |F_N(x, t)|>\lambda \right\}\,.
\end{equation}
To obtain the desired estimate for the level set, let us first state a lemma 
on Weyl's sums. 

\begin{lemma}\label{lem1}
Suppose that $ t\in \mathbb T$ satisfies $|t-a/q|\leq 1/q^2$, where 
$a $ and $q$ are relatively prime. Then if $q\geq N^{d-1}$, 
\begin{equation}\label{weylsum}
\left| \sum_{n=1}^N e^{2\pi i tn^d + 2\pi iP(n)}\right|\leq 
C N^{-d2^{1-d}+1+\e} q^{2^{1-d}}\,. 
\end{equation}
Here $P$ is a real polynomial of degree no more than $d-1$, and the constant $C$ is independent of
$t$, $P$, $a/q$ and $N$. 
\end{lemma}

The proof of Lemma \ref{lem1} relies on Weyl's squaring method. See \cite{Hua} or \cite{M} for details. 
Also we need the following lemma proved in \cite{B1}.

\begin{lemma}\label{lem2}
For any integer $Q\geq 1$ and any integer $n\neq 0$, and any $\varepsilon>0$, 
$$\sum_{ Q\leq q < 2Q}\left|\sum_{a\in\mathcal{P}_q}e^{2\pi  i\frac{a}{q}n}\right|\leq C_\varepsilon
 d(n, Q) Q^{1+\varepsilon}\,.$$
Here $\mathcal P_q $ is given by
\begin{equation}
 \mathcal P_q= \{a\in \mathbb N: 1\leq a\leq q \,\,\, {\rm and}\,\,\, (a,q)=1\}
\end{equation}
and $d(n, Q)$ denotes the number of divisors of $n$ less than $Q$ and $C_\varepsilon$ is a constant independent of $Q, n$. 
\end{lemma}


\begin{proposition}\label{Prop1}
Let $K_N$ be a kernel defined by
\begin{equation}\label{defofKN}
 K_{N}(x, t) =\sum_{n=-N}^N e^{2\pi i t n^d +  2\pi i x n}\,. 
\end{equation} 
For any given positive number $Q$ with $N^{d-1}\leq Q\leq N^d$,
the kernel $K_N$  can be decomposed into $K_{1, Q} + K_{2, Q}$ such that 
\begin{equation}\label{K1}
\|K_{1, Q}\|_\infty \leq C_1 N^{-d2^{1-d}+1+\e} Q^{2^{1-d}}\,.
\end{equation}
and
\begin{equation}\label{K2}
\|\widehat{K_{2, Q}}\|_{\infty} \leq \frac{C_2 N^\varepsilon }{Q}\,.
\end{equation}
Here the constants $C_1, C_2$ are independent of $Q$ and $N$. 
\end{proposition}

\begin{proof}
We can assume that $Q$ is an integer, since  otherwise we can take the integer part of $Q$. 
For a standard bump function $\varphi$ supported on $[1/200, 1/100]$, we set 
\begin{equation}\label{defofPhi}
\Phi(t) = \sum_{ Q\leq q  \leq 5Q}\sum_{a\in\mathcal P_q}\varphi\left(\frac{t-a/q}{1/q^2}\right)\,.
\end{equation}
Clearly $\Phi$ is supported on $[0,1]$. We can extend $\Phi$ to other intervals periodically to obtain 
a periodic function on $\mathbb T$. For this periodic function generated by $\Phi$, we still use $\Phi$ to
denote it.  Then it is easy to see that
\begin{equation}
\widehat{\Phi}(0)=\sum_{q\sim Q}\sum_{a\in\mathcal{P}_q}\frac{\mathcal{F}_{\mathbb R}{\varphi}(0)}{q^2}=\sum_{q\sim Q}\frac{\phi(q)}{q^2}\mathcal{F}_{\mathbb R}{\varphi}(0) \,
\end{equation}
is a constant independent of $Q$. Here $\phi$ is Euler phi function, and $\mathcal{F}_{\mathbb R}$ denotes Fourier transform of a function on $\mathbb R$. Also we have
\begin{equation}\label{Fest}
\widehat{\Phi}(k)= \sum_{q\sim Q}\sum_{a\in\mathcal P_q}\frac{1}{q^2} e^{- 2\pi i \frac{a}{q}k} \mathcal F_{\mathbb R}
\varphi(k/q^2)\,. 
\end{equation}
Applying Lemma \ref{lem2} and the fact that $Q\leq N^d$, we obtain
\begin{equation}\label{Fest1}
\left |\widehat{\Phi}(k)\right|\leq \frac{N^\e}{Q}\,, 
\end{equation}
if $k\neq 0$. 

We now define that 
$$
K_{1, Q}(x, t)=\frac{1}{\widehat{\Phi}(0)}K_N(x, t)\Phi(t), \,\,\,{\rm and}\,\,\, 
K_{2, Q} =  K_N - K_{1, Q}\,. 
$$
(\ref{K1}) follows immediately from Lemma \ref{lem1} since intervals $J_{a/q}=
[\frac{a}{q} + \frac{1}{100q^2}, \frac{a}{q} + \frac{1}{50q^2}]$'s  are pairwise disjoint
for all $Q\leq q\leq 5Q$ and $a\in\mathcal P_q$.

We now prove (\ref{K2}).  In fact, represent $\Phi$ as its Fourier series to get
$$
K_{2, Q}(x, t) = - \frac{1}{\widehat\Phi(0)} \sum_{k\neq 0}\widehat\Phi(k) e^{2\pi i k t} K_N(x, t)\,.
$$
Thus its Fourier coefficient is
$$
\widehat{K_{2, Q}}(n_1, n_2)= -\frac{1}{\widehat\Phi(0)} \sum_{k\neq 0} 
  \widehat\Phi(k){\bf 1}_{\{ n_2 = n_1^d + k\} } (k) 
 \,.
$$
Here $(n_1, n_2)\in \mathbb Z^2$ and ${\bf 1}_A$ is the indicator function of a measurable set $A$.
This implies that 
$\widehat{K_{2, Q}}(n_1, n_2) =0$ if $ n_2 = n_1^d $, and if $n_2\neq n_1^d$, 
$$  \widehat{K_{2, Q}}(n_1, n_2) =  -\frac{1}{\widehat\Phi(0)}
\widehat\Phi(n_2- n_1^d) \,.$$
Applying (\ref{Fest1}), we estimate $\widehat{K_{2, Q}}(n_1, n_2)$ by
$$
 \left|\widehat{K_{2, Q}}(n_1, n_2)\right| \leq \frac{CN^\varepsilon}{Q}\,,
$$
since $N^{d-1}\leq Q\leq N^d$. Henceforth we obtain (\ref{K2}). Therefore we complete the proof. 

\end{proof}

Now we can state our theorem on the level set estimates.  

\begin{theorem}\label{thm2}
For any positive numbers $\e$ and $Q\geq N^{d-1}$, the level set defined as in (\ref{defofElam}) satisfies 
\begin{equation}\label{estE}
\lambda^2 \left| E_\lambda\right|^2 \leq C_1 N^{-d2^{1-d}+1+\e} Q^{2^{1-d}}
\left| E_\lambda\right|^2 + \frac{C_2N^\e}{Q}\left| E_\lambda\right|\,
\end{equation}
holds for all $\lambda>0$.  Here $C_1$ and $C_2$ are constants independent of $N$ and $Q$.  
\end{theorem}

\begin{proof}
Notice that if $Q\geq N^d$, (\ref{estE}) becomes trivial since $E_\lambda=\emptyset$ if $\lambda
\geq CN^{1/2}$. So we can assume that $N^{d-1}\leq Q\leq N^d$. 
For the function $F_N$ and the level set $E_\lambda $ given in (\ref{defofF0}) and (\ref{defofElam}) respectively, we define $f$ to be
$$
f(x, t) = \frac{{F_N(x, t)}}{ |F_N(x, t)|} {\bf 1}_{E_\lambda}(x, t)\,. 
$$
Clearly
$$
\lambda|E_\lambda|\leq \int_{\mathbb T^2} \overline{F_N(x, t)} f(x, t) dx dt\,.  
$$
By the definition of $F_N$, we get
$$
\lambda|E_\lambda|\leq \sum_{n=-N}^N\overline{a_n} \widehat{f}(n, n^d )\,.
$$
Utilizing Cauchy-Schwarz's inequality, we have
$$
\lambda^2|E_\lambda|^2\leq \sum_{n=-N}^N\left| \widehat{f}(n, n^d)\right|^2 \,.
$$
The right hand side can be written as
\begin{equation}
 \langle K_N *f, f\rangle\,. 
\end{equation}
For any $Q$ with $N^{d-1}\leq Q\leq N^d$, we employ Proposition \ref{Prop1} to decompose the kernel $K_N$. We then have 
\begin{equation}
\lambda^2|E_\lambda|^2\leq \left| \langle K_{1, Q} *f, f\rangle \right| 
   +  \left| \langle K_{2, Q} *f, f\rangle \right|  \,
\end{equation}
From (\ref{K1}) and (\ref{K2}), we then obtain 
\begin{equation}\label{estLevel1}
\begin{aligned}
\lambda^2|E_\lambda|^2 &\,\leq \,C_1  N^{-d2^{1-d}+1+\e} Q^{2^{1-d}}\|f\|_1^2 + \frac{C_2 N^\varepsilon}{Q}\|f\|_2^2\\
  &\,\leq \,C_1  N^{-d2^{1-d}+1+\e} Q^{2^{1-d}}|E_\lambda|^2 + \frac{C_2 N^\varepsilon}{Q}|E_\lambda|\,,
\end{aligned}
\end{equation}
as desired. Therefore, we finish the proof of Theorem \ref{thm2}. 
\end{proof}

\begin{corollary}\label{cor1}
If $\lambda\geq 2C_1N^{\frac12-\frac{1}{2^d} +\e}$, then 
\begin{equation}\label{Eest2}
|E_\lambda|\leq \frac{CN^{2^{d-1}-d+\e}}{\lambda^{2^d+2}}\,. 
\end{equation}
Here $C_1$ is the constant $C_1$ in Theorem \ref{thm2} and $C$ is a constant independent of $N$ and $\lambda$.  
\end{corollary}

\begin{proof}
Since $\lambda\geq 2C_1N^{\frac12-\frac{1}{2^d}+\e}$, we simply take $Q$ satisfies 
$ 2C_1 N^{-d2^{1-d}+1+\e} Q^{2^{1-d}} = \lambda^2  $.
Then Corollary \ref{cor1} follows from Theorem \ref{thm2}. 
\end{proof}

\begin{remark}
Corollary \ref{cor1} is also true even if $n^d$ in (\ref{defofF0}) is replaced by $n^d+P(n)$, where
$P$ is a polynomial in $\mathbb Z[x]$ whose degree is no more than $d-1$.   
\end{remark}

We now are ready to finish the proof of Theorem \ref{thm1}. 
We only prove the case when $d$ is odd. The even case can be done similarly by using 
$A_{d, 4, N}\leq C$.  In fact, let $p\geq (d-2)2^d+6 $ and 
write $\|F\|_p^p$ as
\begin{equation}\label{Fsplit}
 p\int_0^{2C_1N^{\frac12-\frac{1}{2^d} +\e}}\lambda^{p-1} |E_\lambda|d\lambda  +
  p\int_{2C_1N^{\frac12-\frac{1}{2^d} +\e}}^{2N^{1/2}}\lambda^{p-1} |E_\lambda|d\lambda  \,.  
\end{equation}
Observe that $ A_{d, 6, N}\leq  N^\e $ implies 
\begin{equation}
 |E_\lambda|\leq \frac{N^\e}{\lambda^6}\,.  
\end{equation}
Thus the first term in (\ref{Fsplit}) is bounded by
\begin{equation}
  C N^{ (\frac12-\frac{1}{2^d})(p-6) +\e } \leq CN^{\frac{p}{2}-(d+1) +\e}\,,
\end{equation} 
since $p\geq  (d-2)2^d+6 $. 
From (\ref{Eest2}), the second term is majorized  by
\begin{equation}
 C N^{\frac{p}{2}-(d+1)+\e}\,.
\end{equation}
Putting both estimates together, we complete the proof of Theorem \ref{thm1}. \\

\section{A Lower bound of $K_{d, p, N}$}\label{Lower}
\setcounter{equation}0

In this section we show that $N^{\frac12-\frac{d+1}{p}}$ is the best upper bound of $K_{d, p, N}$ if $p\geq 2(d+1)$.
Hence (\ref{res-est}) can not be improved substantially, and it is sharp up to a factor of $N^\e$. \\

For $b\in \mathbb N$, let $S(N; b)$ be defined by
\begin{equation}\label{defJNb}
S(N; b)=\int_{\mathbb T^2} \left| \sum_{n=-N}^N e^{2\pi i tn^d + 2\pi i x n} \right|^{2b} dx dt\,.
\end{equation}

\begin{proposition}\label{prop2}
Let $S(N;b)$ be defined as in (\ref{defJNb}). Then
\begin{equation}\label{lbofJ}
S(N;b)\geq C \left(N^b + N^{2b-(d+1)}\right) \,.
\end{equation}
Here $C$ is a constant independent of $N$. 
\end{proposition}

\begin{proof}
The proof is based on a standard argument in additive number theory.
Clearly $S(N;b)$ is equal to the number of solutions of 
\begin{equation}\label{sys}
\begin{cases}
 n_1+\cdots +n_b   =   m_1+\cdots + m_b \\
 n_1^d+\cdots + n_b^d= m_1^d+\cdots +m_b^d  \,
\end{cases}
\end{equation}
with $n_j, m_j\in \{-N, \cdots, N\}$ for all $j\in\{1, \cdots, b\}$. 
For each $(m_1, \cdots, m_b)$, we may obtain a solution of (\ref{sys}) by taking 
$(n_1, \cdots, n_b)=(m_1, \cdots, m_b)$. Thus 
\begin{equation}\label{J-est}
S(N; b)\geq N^b\,. 
\end{equation}
To derive a further lower bound for $S(N; b)$, we set $\Omega$ to be
\begin{equation}\label{defofOme}
\Omega = \left\{(x, t):  |x|\leq \frac{1}{60 N}\,,\,\,\,\,  |t|\leq \frac{1}{60 N^d}\right\}\,. 
\end{equation}
If $(x, t)\in\Omega$ and $|n|\leq N$, then 
\begin{equation}
 \left| tn^d+xn\right|\leq \frac{1}{30}\,. 
\end{equation}
Henceforth if $(x, t)\in\Omega$, 
\begin{equation}
\left| \sum_{n=-N}^N e^{2\pi i tn^d + 2\pi i x n} \right|\geq \left| {\rm Re}\sum_{n=-N}^N e^{2\pi i tn^d +  2\pi i x n} \right| \geq \sum_{n=-N}^{N}\cos \left((2\pi tn^d+2\pi xn) \right)\geq CN\,.
\end{equation}
Consequently, we have 
\begin{equation}\label{Jest2}
S(N; b)\geq \int_{\Omega}\left| \sum_{n=-N}^N e^{ 2\pi i tn^d +  2\pi i x n} \right|^{2b}dx dt\geq 
 CN^{2b}|\Omega|\geq CN^{2b-(d+1)}\,. 
\end{equation}
\end{proof}

\begin{proposition}\label{prop3}
Let $p\geq 2$ be even. Then $K_{d, p, N}$ satisfies 
\begin{equation}\label{LbofA}
  K_{d, p, N} \geq C\left(1+ N^{\frac12-\frac{d+1}{p}}\right)\,.
\end{equation}
Here $C$ is a constant independent of $N$. 
\end{proposition}

\begin{proof}
Let $p=2b$ since $p$ is even. Setting $a_n=1$ for all $n$ in the definition of $K_{d, p, N}$, we get
\begin{equation}
  S(N; b) \leq  K_{d, p, N}^p (2N)^{b}\,. 
\end{equation}
Consequently, by Proposition {\ref{prop2}}, we conclude (\ref{LbofA}). 
\end{proof}

\section{Estimates of $S(N;b)$}
\setcounter{equation}0

We have the following estimates for $S(N;b)$. The $d=3$ case 
was proved by Hua. The method of Hua is different from what we utilize in this paper. 

\begin{theorem}\label{thmS-est} 
Let $S(N; b)$ be defined as in (\ref{defJNb}) and $d\geq 3$ be odd. Then 
\begin{equation}\label{S-est}
S(N;b)\leq CN^{2b-(d+1)+\e}\,
\end{equation}
holds provided  $ b\geq \max\{ 2^{d-1}+1 , 2^{d-2}(d-5)+3\}$.
\end{theorem}

By Proposition \ref{prop2}, we see that the estimate (\ref{S-est}) is (almost) sharp.  
The desired upper bound for $S(N; {d+1})$ is not yet obtained. 
We now prove Theorem \ref{thmS-est}. 

\begin{proof}
Let $G_\lambda$ be the level set given by
\begin{equation}\label{defofG0}
 G_\lambda = \left\{ (x, t)\in \mathbb T^2: |K_N(x,t)|\geq \lambda  \right\}\,.
\end{equation}
Here $K_N$ is the function defined as in (\ref{defofKN}). \\

let $f=\Id_{G_\lambda}K_N/|K_N|$ and we then have
\begin{equation}\label{estGLam}
\lambda|G_\lambda| \leq \sum_{n=-N}^N \widehat f(n, n^d) = \langle f_{N}, K_N \rangle\,,
\end{equation} 
where $ f_N$ is a rectangular Fourier partial sum defined by 
\begin{equation}
f_N(x, t) = \sum_{\substack{|n_1|\leq N \\ |n_2|\leq N^d } } \wh f(n_1, n_2) e^{2\pi n_1 x } e^{2\pi i n_2 t}\,.
\end{equation}

Employing Proposition \ref{Prop1} for $K_N$, we estimate
the level set $G_\lambda $ by
\begin{equation}
\lambda|G_\lambda|\leq |\langle f_{N}, K_{1, Q}\rangle| + |\langle f_{N}, K_{2, Q}\rangle |  \,
\end{equation}
for any $ Q\geq N^{d-1}$.
From (\ref{K1}) and (\ref{K2}),  $\lambda|G_\lambda| $ can be bounded further by
\begin{equation}
C\left( N^{-d2^{1-d}+1+\e} Q^{2^{1-d}}
\|f_N\|_1 + \sum_{
\substack{ |n_1|\leq N\\ |n_2|\leq N^d  }  }
   \left| \wh{K_{2, Q}}(n_1, n_2)\wh f(n_1, n_2) \right|  \right)\,.
\end{equation}
Thus from the fact that $L^1$ norm of Dirichlet kernel $D_N$ is comparable to $\log N$, (\ref{K2}),  and Cauchy-Schwarz inequality, we have 
\begin{equation}
\lambda|G_\lambda|\leq C N^{-d2^{1-d}+1+\e} Q^{2^{1-d}}   |G_\lambda| + 
  \frac{C N^{\frac{d+1}{2}+\e}  }{Q}|G_\lambda|^{1/2}\,,
\end{equation}
for all $Q\geq N^{d-1}$. 
For $\lambda\geq  2C N^{1-2^{1-d} +\e}$, take $Q$ to be a number satisfying  
$$ 2C N^{-d2^{1-d}+1+\e} Q^{2^{1-d}}  = \lambda $$ 
and then we obtain
\begin{equation}\label{estofG}
 |G_\lambda| \leq \frac{CN^{2^d-d+1 } }{\lambda^{2^d+2}}\,.
\end{equation}
Notice that 
\begin{equation}\label{L2ofS}
\|K_N\|_6 \leq   N^{\frac{1}{2}} K_{d, 6, N}\leq N^{\frac{1}{2}+\e}\,. 
\end{equation}
Henceforth by (\ref{estGLam}) we majorize $|G_\lambda|$ by 
\begin{equation}\label{estofG2}
|G_\lambda| \leq \frac{CN^{3+\e}}{\lambda^6} \,.
\end{equation}
For $b\geq 2^{d-1}+1$, we now estimate $S(N; p)$ by
\begin{equation}\label{JN5est}
S(N;b)\leq C\int_{2C N^{1-2^{1-d} +\e}}^{2 N}
 \lambda^{2b -1 }|G_\lambda| d\lambda 
+ C\int_0^{2C N^{1-2^{1-d} +\e}}
 \lambda^{2b-1 }|G_\lambda| d\lambda  \,. 
\end{equation}
From (\ref{estofG}),  the first term in the right hand side 
of (\ref{JN5est}) can be bounded by $ CN^{2b-d-1+\e}$. From (\ref{estofG2}), the second term is clearly bounded by
$N^{2b-d-1+\e}$.  Putting both estimates together,
\begin{equation}\label{estofJnorm}
S(N; b)\leq CN^{2b-(d+1)+\e}\,,
\end{equation}
as desired. Therefore, we complete the proof.
\end{proof}

\section{Estimates for the nonlinear term }\label{nonlinear}
\setcounter{equation}0

For any measurable function $u$ on $\mathbb T\times \mathbb R$, we define the space-time Fourier transform by
\begin{equation}\label{defUhat}
\wh{u}(n, \lambda) = \int_\mathbb R \int_{\mathbb T}u(x, t) e^{- i n x} e^{- i \lambda t} dx \,dt\,
\end{equation} 
and set 
$$\langle x\rangle:= 1+|x|\,.$$

We now introduce the $X_{s, b}$ space, initially used by Bourgain.

\begin{definition}\label{defofXsb}
Let $I$ be an time interval in $\mathbb R$ and $s, b\in\mathbb R$.  Let $X_{s, b}(I)$ be the space of functions $u$ 
on $ \mathbb T\times I $ that may be represented as
\begin{equation}
u(x,t) = \sum_{n\in \mathbb Z}\int_{\mathbb R} \wh u(n, \lambda) e^{ i nx} e^{ i \lambda t} d\lambda\,\,\,
{\rm for}\,\,\,\, (x, t)\in \mathbb T\times I\,
\end{equation}
with the space-time Fourier transform $\wh u$ satisfying
\begin{equation}
\|u\|_{X_{s,b}(I)} =  \left ( \sum_n\int\langle n\rangle^{2s}\langle \lambda + n^5\rangle^{2b}|\widehat{u}(n, \lambda)|^2d\lambda  \right)^{1/2} <\infty\, .
\end{equation}
Here the norm should be understood as a restriction norm. 
\end{definition}

We should take the time interval to be $[0, \delta]$ for a small positive number $\delta$, 
and abbreviate $\|u\|_{X_{s, b}(I)}$ as $\|u\|_{s, b}$ for any function $u$ restricted to 
$ \mathbb T\times [0, \delta]$.   We also
define
\begin{equation}\label{defofNorm}
\|u\|_{Y_s}:= \|u\|_{s, \frac{1}{2}} +
\left(\sum_n\langle n\rangle^{2s}\left(\int\left|\widehat{u}(n,\lambda)\right|d\lambda\right)^2\right)^\frac{1}{2}\,.
\end{equation}

Let $\psi$ be a bump function supported in $[-2, 2]$ with $\psi(t)=1, |t|\leq 1$, and let $\psi_\delta$ be 
$$
\psi_\delta(t) =\psi(t/\delta)\,.
$$
For any $w$ which is a nonlinear function of $u$, the nonlinear operator $\mathcal N$ is given by
\begin{equation}
 \mathcal N u = -\psi_\delta(t)\int_0^te^{-(t-\tau)\partial_x^5}w(x,\tau)d\tau.
\end{equation}

\begin{lemma}\label{estNu}
The nonlinear term $ \mathcal N$ satisfies 
\begin{equation}\label{estNu1}
\|\mathcal N u\|_{Y_s}\leq C \left( \|w\|_{s, -\frac12} + \left(\sum_n\langle n\rangle^{2s}\left(\int\frac{|\widehat{w}(n,\lambda)|}{\langle\lambda+n^5\rangle}d\lambda\right)^2\right)^\frac{1}{2}  \right)\,,
\end{equation}
where $C$ is a constant independent of $\delta$. 
\end{lemma}

\begin{proof}
Represent $w$ as its space-time inverse Fourier transform so that 
we write 
\begin{equation}
\mathcal Nu(x,t) = -\psi_\delta(t)\int_0^t e^{-(t-\tau)\partial_x^5}\left(\sum_n\int\widehat{w}(n,\lambda)e^{inx}e^{i\lambda\tau}d\lambda\right)d\tau\,,
\end{equation} 
which is equal to
\begin{align*}
 & -\psi_\delta(t)\sum_n\int \widehat{w}(n,\lambda)\int_0^t e^{-(t-\tau)( in)^5}e^{inx}e^{i\lambda\tau}d\tau d\lambda\\
=&-\psi_\delta(t) \sum_n\int \widehat{w}(n,\lambda) e^{ inx}e^{-in^5t}\ \frac{e^{i(\lambda+n^5)t}-1}{i(\lambda+n^5)}\ d\lambda\,.
\end{align*}
We decompose the nonlinear term $\mathcal Nu$ into three parts, denoted by $\mathcal N_1, \mathcal N_2, 
\mathcal N_3$ respectively. 

\begin{align*}
\mathcal Nu(x,t)
=& -\psi_\delta (t)\sum_n\int_{|\lambda+n^5|\leq\frac{1}{100\delta}}\widehat{w}(n,\lambda)e^{inx}e^{-in^5t}\sum_{k\geq1}\frac{(it)^k}{k!}(\lambda+n^5)^{k-1}d\lambda\\
& +i \psi_\delta(t)\sum_n\int_{|\lambda+n^5| > \frac{1}{100\delta}}\frac{\widehat{w}(n,\lambda)}{\lambda+n^5}e^{inx}e^{i\lambda t} d\lambda\\
&- i\psi_\delta(t)\sum_n\left(\int_{|\lambda+n^5| >\frac{1}{100\delta}}\frac{\widehat{w}(n,\lambda)}{\lambda+n^5}d\lambda\right)e^{inx}e^{-in^5t}\\
:=& \mathcal N_1u + \mathcal N_2u + \mathcal N_3u.
\end{align*}

First we estimate $\mathcal N_2$.  Using Fourier series expansion for $\psi$, we get
$$
\psi_\delta(t) =\sum_{m\in \mathbb Z} C_m e^{im t/\delta}\,.
$$
Here the coefficients $C_m$'s satisfy 
$$
 C_m \leq C (1+|m|)^{-100}\,. 
$$
Hence $\mathcal N_2u$ can be represent as
\begin{equation}
\mathcal N_2 u =i\sum_{m}C_m \sum_{n} e^{inx}\int_{|\lambda+n^5| > \frac{1}{100\delta}}\frac{\widehat{w}(n,\lambda)}{\lambda+n^5} e^{i(\lambda + m/\delta ) t} d\lambda
\end{equation}
By a change of variables $(\lambda+m/\delta)\mapsto \lambda $, 
\begin{equation}
\mathcal N_2 u =i\sum_{m}C_m \sum_{n} e^{inx}\int_{|\lambda-\frac{m}{\delta}+n^5| > \frac{1}{100\delta}}\frac{\widehat{w}(n,\lambda-m/\delta)}{\lambda-\frac{m}{\delta}+n^5} e^{i\lambda t} d\lambda
\end{equation}

Thus we estimate
\begin{equation}
\|\mathcal N_2 u\|_{s, \frac12}^2\leq  C\sum_m(1+|m|)^{-50}\sum_n \langle n\rangle^{2s}\int_{
 |\lambda-\frac{m}{\delta}+n^5|>\frac{1}{100\delta}} \frac{\langle \lambda+n^5\rangle \left| \wh w(n, \lambda-m/\delta) \right|^2}{|\lambda-\frac{m}{\delta}+n^5|^2} d\lambda\,.
\end{equation}
Changing variables again, we obtain
\begin{equation}
\|\mathcal N_2 u\|_{s, \frac12}^2\leq  C\sum_m(1+|m|)^{-50}\sum_n \langle n\rangle^{2s}\int_{
 |\lambda+n^5|>\frac{1}{100\delta}} \frac{\langle \lambda+\frac{m}{\delta}+n^5\rangle \left| \wh w(n, \lambda) \right|^2}{ \langle 
 \lambda+n^5\rangle^2} d\lambda\,.
\end{equation}
Notice that $ |\lambda+n^5|> \frac{1}{100\delta} $ implies 
\begin{equation}
 \langle \lambda +\frac{m}{\delta} +n^5\rangle \leq  200m\langle \lambda +n^5\rangle\,.
\end{equation}
We obtain immediately 
\begin{equation}\label{estN21}
\|\mathcal N_2 u\|_{s, \frac12}\leq C\| w\|_{s, -\frac{1}{2}}\,.
\end{equation}
On the other hand, 
$$
\sum_n \langle n\rangle^{2s}\!\left( \!\int  {| \wh{\mathcal N_2u}(n, \lambda)|}  d\lambda\right)^2\!\!
\!\!\leq C\sum_{m}\langle m\rangle^{-5}\!\sum_n\langle n \rangle^{2s}\!\!\left(\! \int_{|\lambda-\frac{m}{\delta}+n^5|>\frac{1}{100\delta}} \!\!  \frac{|\wh w(n, \lambda-m/\delta) |d\lambda}{  |\lambda-\frac{m}{\delta}+n^5|}  \right)^2,
$$
which is clearly bounded by
\begin{equation}\label{estN22}
\sum_n\langle n \rangle^{2s}\left(\int \frac{|\wh w(n, \lambda) |d\lambda}{  \langle\lambda+n^5\rangle}  \right)^2.
\end{equation}
Putting (\ref{estN21}) and (\ref{estN22}) together, we have 
\begin{equation}\label{estN2}
\| \mathcal N_2 u \|_{Y_s}\leq C\left ( \|w\|_{s,-\frac{1}{2}}+\left(\sum_n\langle n\rangle^{2s}\left(\int\frac{|\widehat{w}(n,\lambda)|}{\langle\lambda+n^5\rangle}d\lambda\right)^2\right)^\frac{1}{2} \right)\,.
\end{equation}

We now estimate $\mathcal N_1$.
Let $A_n$ be defined by
\begin{equation}
A_n = \int_{|\lambda+n^5|\leq \frac{1}{100\delta}}\wh w(n, \lambda)(\lambda+n^5)^{k-1}d\lambda\,.
\end{equation}
Then $\mathcal N_1u$ can be written as
\begin{equation}
\mathcal N_1u(x, t) =-\sum_{k\geq 1}\frac{i^k}{k!}t^k\psi_\delta(t)\sum_n  A_n e^{inx}e^{-in^5t}\,.
\end{equation}
Hence the space-time Fourier transform of $\mathcal N_1u$ satisfies 
\begin{equation}\label{N1Fest}
\left|\wh {\mathcal N_1u}(n, \lambda)\right|\leq \sum_{k\geq 1}\frac{1}{k!} |A_n|\left| \mathcal F_{\mathbb R}(\tilde{\psi_\delta})(\lambda+n^5)\right|\,,
\end{equation}
where $\tilde{\psi_\delta}(t)=t^k\psi_\delta(t)$.  Using the definition of Fourier transform, we have
$$
\left| \mathcal F_{\mathbb R}(\tilde{\psi_\delta})(\lambda+n^5)\right|\leq C\delta^{k+1}k^{3}\langle  \delta (\lambda+n^5)\rangle^{-3}\,.
$$
Thus 
\begin{eqnarray*}
\|\mathcal N_1u\|_{Y_s}^2 & \leq & 
 \sum_{k\geq 1} \frac{C}{k^5}\sum_n \langle n\rangle^{2s}|A_n|^2\delta^{2k}\int  \delta^2\langle \lambda+n^5\rangle
 \langle  \delta (\lambda+n^5)\rangle^{-6}
 d\lambda  \\
 &   &+  \sum_{k\geq 1}\frac{C}{k^5}\sum_n \langle n\rangle^{2s}|A_n|^2\delta^{2k}\left(\int  \delta \langle  \delta (\lambda+n^5)\rangle^{-3}  d\lambda \right)^2\\
&\leq &\sum_{k\geq 1} \frac{C}{k^5}\sum_n \langle n\rangle^{2s}|A_n|^2\delta^{2k}  \,.
\end{eqnarray*}
Clearly $A_n$ is bounded by
\begin{equation}
|A_n|\leq C\delta^{-k}\int \frac{|\wh w(n, \lambda)|}{\langle\lambda+n^5\rangle}d\lambda\,.
\end{equation}
Henceforth, we obtain
\begin{equation}\label{N1est}
\|\mathcal N_1u\|_{Y_s}\leq C\left(\sum_n\langle n\rangle^{2s}\left(\int\frac{|\widehat{w}(n,\lambda)|}{\langle\lambda + n^5 \rangle}d\lambda\right)^2\right)^\frac{1}{2}.
\end{equation}

Similarly,  we may obtain
\begin{equation}\label{N3est}
\|\mathcal N_3u\|_{Y_s}\leq C\left(\sum_n\langle n\rangle^{2s}\left(\int\frac{|\widehat{w}(n,\lambda)|}{\langle\lambda +n^5\rangle}d\lambda\right)^2\right)^\frac{1}{2}.
\end{equation}

Therefore we complete the proof.

\end{proof}

\vspace{0.6cm}

\section{Local well-posedness of (\ref{gKdV})}\label{LWP1}
\setcounter{equation}0

We now start to derive the local well-posedness of (\ref{gKdV}).  For this purpose, we only need to 
consider the well-posedness of the Cauchy problem:
\begin{equation}\label{KdV2}
\begin{cases}
u_t+ \p_x^5u+\left(u^k-\int_\mathbb{T}u^kdx\right)u_x=0\\
u(x,0)=\phi(x),\qquad x\in\mathbb{T},\ t\in\mathbb{R}
\end{cases}.
\end{equation}
Here $k\geq 2$ and we only need to consider the monomial case without loss of 
generality. 
 This is because the gauge transform 
\begin{equation}\label{Gauge1}
u(x,t):= v\left(x-\int_0^t\int_\mathbb{T}v^k(y,\tau)dyd\tau,  t  \right)\,
\end{equation}
can be used here for reducing the well-posedness problem of (\ref{gKdV}) to 
the well-posedness of (\ref{KdV2}). This gauge transform was employed in \cite{Tao}.\\

Let $w$ be the nonlinear function defined by
\begin{equation}\label{defofw}
w = \left( u^k - \int u^k dx\right)u_x\,.
\end{equation}
We need the following estimate 
on the nonlinear function $w$, in order to establish a contraction on the space $\{u: \|u\|_{Y_s}\leq M\}$ for some 
$M>0$. We postpone the proof of Proposition \ref{propofw} to Section \ref{proofofw}.

\begin{proposition}\label{propofw}
For $s>1/2$, there exists $\theta>0$ such that, for the nonlinear function $w$ 
 given by (\ref{defofw}),  
\begin{equation}
\|w\|_{s, -\frac12} + \left(\sum_n\langle n\rangle^{2s}\left(\int\frac{|\widehat{w}(n,\lambda)|}{\langle\lambda+n^5\rangle}d\lambda\right)^2\right)^\frac{1}{2} \leq C\delta^\theta \|u\|_{Y_s}^{k+1}.
\end{equation}
Here $C$ is a constant independent of $\delta$ and $u$. 
\end{proposition}

By applying Duhamel principle, the corresponding integral equation associated to (\ref{KdV2}) is
\begin{equation}
u(x,t) = e^{-t\partial_x^5}\phi(x)-\int_0^te^{-(t-\tau)\partial_x^5}w(x,\tau)d\tau,
\end{equation}
where $w$ is defined as in (\ref{defofw}).\\

Since we are only seeking for the local well-posedness, we may use a bump function to truncate time variable. Then it suffices to find a local solution of
$$u(x,t) = \psi_\delta(t)e^{-t\partial_x^5}\phi(x)-\psi_\delta(t)\int_0^te^{-(t-\tau)\partial_x^5}w(x,\tau)d\tau.$$
Let $T$ be an operator given by
\begin{equation}\label{defofT}
Tu(x,t):= \psi_\delta(t)e^{-t\partial_x^5}\phi(x)-\psi_\delta(t)\int_0^te^{-(t-\tau)\partial_x^5}w(x,\tau)d\tau.
\end{equation}
The first term (the linear term) and the second term
(the nonlinear term)in (\ref{defofT}) are denoted by ${\mathcal L}u$ and ${\mathcal N}u$, 
respectively.  Henceforth $Tu$ can be represented as $ {\mathcal L}u +{\mathcal N}u $.    

\begin{lemma}\label{estLu}
The linear term $ \mathcal L$ satisfies 
\begin{equation}\label{estLu1}
\|\mathcal L u\|_{Y_s}\leq C\|\phi\|_{H^s}\,. 
\end{equation}
Here $C$ is a constant independent of $\delta$. 
\end{lemma}

\begin{proof}
Notice that
$$\widehat{\mathcal L u}(n,\lambda) = \widehat{\phi}(n){\mathcal F}_{\mathbb R}{\psi_\delta}(\lambda+n^5) =
 \widehat{\phi}(n) \delta \mathcal F_{\mathbb R}{\psi}\left( \delta( \lambda+n^5)\right)  ,$$
Thus from the definition of $Y_s$ norm, 
\begin{align*}
\|\mathcal Lu\|_{Y_s} = &\left(\sum_n \int \langle n\rangle^{2s}\langle \lambda+n^5\rangle \left|\widehat{\phi}(n)
  \delta \mathcal F_{\mathbb R}{\psi}\left(\delta(\lambda+n^5)\right)\right|^2 d\lambda\right)^\frac{1}{2}\\
&+ \left(\sum_n \langle n\rangle^{2s}\left(\int \left|\widehat{\phi}(n) \delta \mathcal F_{\mathbb R}\psi\left(\delta(\lambda+n^5)\right) \right| d\lambda \right)^2\right)^\frac{1}{2}\,.
\end{align*}
Since $\psi$ is a Schwartz function, its Fourier transform is also a Schwartz function.  Using the fast decay property 
for the Schwartz function, we have 
$$
\|\mathcal L u\|_{Y_s}\leq C \left(\sum_n \langle n \rangle^{2s}\left|\widehat{\phi}(n)\right|^2\right)^\frac{1}{2} =C \|\phi\|_{H^s}.
$$

\end{proof}

\begin{proposition}\label{propTu}
Let $s>1/2$ and $T $ be the operator defined as in (\ref{defofT}).  Then there exits a
positive number $\theta$ such that
\begin{equation}\label{estofTu1}
\|Tu\|_{Y_s}\leq C\left(\|\phi\|_{H^s} + \delta^\theta \|u\|_{Y_s}^{k+1}\right)\,. 
\end{equation}
Here $C$ is a constant independent of $\delta$. 
\end{proposition}

\begin{proof}
Since $Tu=\mathcal L u +\mathcal N u$, Proposition \ref{propTu} follows from
Lemma {\ref{estLu}}, Lemma {\ref{estNu}} and Proposition {\ref{propofw}}.   
\end{proof}

Proposition \ref{propTu} yields that for $\delta$ sufficiently small, $T$ maps a ball in $Y_s$ 
into itself. Moreover, we write 
\begin{eqnarray*}
& & \left(u^k-\int_\mathbb{T}u^k dx \right)u_x- \left(v^k-\int_\mathbb{T}v^kdx\right)v_x\\
& =&
\left(u^k-\int_\mathbb{T}u^k dx \right)(u-v)_x + \left((u^k-v^k)-\int_\mathbb{T}(u^k-v^k) dx \right)v_x\\
\end{eqnarray*}
which equals to 
\begin{equation}\label{repu-v}
\left(u^k-\int_\mathbb{T}u^k dx \right)(u-v)_x + \sum_{j=0}^{k-1}
\left((u-v)u^{k-1-j}v^j-\int_\mathbb{T}(u-v)u^{k-1-j}v^j dx \right)v_x\,.
\end{equation}
For $k+1$ terms in (\ref{repu-v}), repeating similar argument as in the proof of Proposition {\ref{propofw}},  
one obtains, for $s>1/2$, 
\begin{equation}\label{contra}
\|Tu-Tv\|_{Y_s}\leq C\delta^\theta \left( \|u\|^k_{Y_s} +\sum_{j=1}^{k-1}\|u\|_{Y_s}^{k-1-j}\|v\|_{Y_s}^{j+1}\right) \|u-v\|_{Y_s}\,.
\end{equation}
Henceforth, for $\delta>0$ small enough, $T$ is a contraction and the local well-posedness 
follows from Picard's fixed-point theorem.  \\

\section{Proof of Proposition \ref{propofw}}\label{proofofw}
\setcounter{equation}0

From the definition of $w$ in (\ref{defofw}), we may write $\wh w(n, \lambda)$ as
\begin{equation}\label{w}
\sum_{\substack{m+n_1+\cdots+n_k=n\\n_1+\cdots+n_k\neq0}}m\int\widehat{u}(m,\lambda-\lambda_1-\cdots-\lambda_k)
\widehat{u}(n_1,\lambda_1)\cdots\widehat{u}(n_k,\lambda_k)d\lambda_1\cdots d\lambda_k.
\end{equation}

By duality, there exists a sequence 
 $\{A_{n, \lambda}\}$ satisfying 
\begin{equation}\label{AnL}
\sum_{n\in\mathbb Z} \int_{\mathbb R}|A_{n, \lambda}|^2 d\lambda \leq 1\,,
\end{equation}
and $\|w\|_{s, -\frac12}$ is bounded by
\begin{equation}\label{ws1}
\sum_{\substack{m+n_1+\cdots+n_k=n\\n_1+\cdots+n_k\neq0}}\!\int\frac{\langle n\rangle^s |m|}{\langle \lambda +n^5 \rangle^\frac{1}{2}}|\widehat{u}(m,\lambda-\lambda_1-\cdots-\lambda_k)|\cdot
 |\widehat{u}(n_1,\lambda_1)|\cdots|\widehat{u}(n_k,\lambda_k)||A_{n,\lambda}|
d\lambda_1\cdots d\lambda_k d\lambda.
\end{equation}

Since the $X_{s, b}$ is a restriction norm, we may assume that $u$ is supported in $\mathbb T\times[0, \delta]$. Moreover, we may assume that $|\wh u|^\vee$ is supported in a $\delta$-sized time interval
(see \cite{H-L}).  Without loss of generality we can also assume $|n_1|\geq|n_2|\geq\cdots\geq|n_k|$.\\

The trouble occurs mainly because of the factor $ |m|$ resulted from $\p_x u$. The idea is that either 
the factor $\langle \lambda+n^5 \rangle^{-\frac{1}{2}}$ can be used to cancel $|m|$, or $|m|$ can be 
distributed to some of $ \wh u$'s. More precisely, we consider three cases. 
\begin{eqnarray}
  &   |m| < 1000k^2|n_2|\,; &  \label{case1}\\
  &  1000k^2|n_2|\leq  |m| \leq  100k|n_1| \,; &\label{case2}\\
 &    |m|> 100k|n_1|\,. & \label{case3}
\end{eqnarray}

\subsection{Case (\ref{case1})}

This is the simplest case.   In fact, 
in this case,  it is easy to see that
\begin{equation}\label{shift1}
\langle n\rangle^s |m|  \leq   C \langle n_1\rangle^s \langle n_2\rangle^\frac{1}{2} \langle m\rangle^\frac{1}{2}.
\end{equation}
Let 
\begin{eqnarray}
& &F(x,t) = \sum_n \int \frac{|A_{n, \lambda}|}{\langle \lambda+n^5\rangle^{\frac12}} e^{i\lambda t}  e^{inx} d\lambda\,;
  \label{defofF} \\
& &
G(x,t) = \sum_n \int \langle n\rangle^\frac{1}{2}|\widehat{u}(n,\lambda)|    e^{i\lambda t}  e^{inx} d\lambda\,
\label{defofG}\\
& &
H(x,t) = \sum_n \int \langle n\rangle^s  |\wh u(n,\lambda)|    e^{i\lambda t}  e^{inx} d\lambda\, \label{defofH}\\
& &
U(x,t) = \sum_n \int    |\wh u(n,\lambda)|    e^{i\lambda t}  e^{inx} d\lambda\,
\label{defofU}
\end{eqnarray}

Then using (\ref{shift1}),  we can estimate (\ref{ws1}) by
\begin{equation}\label{estWs1}
C\!\!\!\!\!\sum_{m+n_1+\cdots+n_k=n}\!\int\!\widehat{F}(n,\lambda)\widehat{G}(m,\lambda-\lambda_1-\cdots-\lambda_k)\widehat{H}(n_1,\lambda_1)\widehat{G}(n_2,\lambda_2)
\prod_{j=3}^k \wh U(n_j, \lambda_j) d\lambda_1\cdots d\lambda_k d\lambda\,,
\end{equation}
which clearly equals 
$$
C \int_{\mathbb{T}\times\mathbb{R}} F(x,t)G(x,t)^2H(x,t)U(x,t)^{k-2}dxdt \,.
$$
Apply H\"older inequality to majorize it by
$$
C\|F\|_4\|G\|_{6+}^2\|H\|_4\|U\|_{6(k-2)-}^{k-2}\,.
$$
Since $U$ is supported on $\mathbb T\times [-2\delta, 2\delta]$, one more use of H\"older inequality yields 
\begin{equation}
(\ref{ws1}) \leq C\delta^\theta\|F\|_4\|G\|_{6+}^2\|H\|_4\|U\|_{6(k-2)}^{k-2}\,.
\end{equation}

We list some useful local embedding facts on $X_{s, b}$. 
\begin{eqnarray}
& &X_{0,\frac{3}{10}}\subseteq L_{x,t}^4 \,,\,\,\,\, X_{0+, \frac{1}{2}+}\subseteq L^{6}_{x,t} \,,\,  \,\,\,  \,(t\ \text{local} ) \label{emb1} \\
& &X_{\alpha,\frac{1}{2}}\subseteq L_{x,t}^q,\quad 0<\alpha<\frac{1}{2},\ 2\leq q<\frac{6}{1-2\alpha}\quad   (t\ \text{local}), \label{emb2}\\
& &X_{\frac{1}{2}-\alpha, \frac{1}{2}-\alpha}\subseteq L_t^qL_x^r,\quad 0<\alpha<\frac{1}{2},\ 2\leq q, r<\frac{1}{\alpha} \label{emb3}.
\end{eqnarray}
The two embedding results in (\ref{emb1}) are consequences of the discrete restriction estimates on $L^4$  (Theorem \ref{B4est}) and $L^6$ (Theorem \ref{thm0}) respectively. (\ref{emb2}) and (\ref{emb3}) follow by interpolation.
(\ref{emb1}) yields 
$$
\|F\|_4\leq C\|F\|_{0, \frac{3}{10}} \leq C\left(  \sum_n \int |A_{n, \lambda}|^2 d\lambda\right)^{1/2}\leq C\,,
$$
and 
$$
\|H\|_4\leq C\|H\|_{0,\frac{3}{10}}\leq C\|u\|_{s,\frac{1}{2}}\leq C\|u\|_{Y_s}\,.
$$
(\ref{emb2}) implies 
$$
\|G\|_{6+}\leq C\|G\|_{0+,\frac{1}{2}}\leq C\|u\|_{s,\frac{1}{2}}\leq C\|u\|_{Y_s}\,.
$$
Using (\ref{emb3}), we get
$$
\|U\|_{6(k-2)}\leq C\|U\|_{\frac{1}{2}-,\frac{1}{2}-}\leq C\|u\|_{s,\frac{1}{2}}\leq C\|u\|_{Y_s}.
$$ 
Henceforth, we have, for the case (\ref{case1}),
\begin{equation}\label{case1est}
(\ref{ws1})\leq C\delta^\theta \|u\|_{Y_s}^{k+1}\,. 
\end{equation}

\subsection{Case (\ref{case2})}\label{Hicase2}

In this case, we should further consider two subcases.
\begin{eqnarray}
  & |m+n_1| \leq  1000k^2|n_2|&   \label{subcase21}\\
&|m+n_1| > 1000k^2|n_2|   &  \label{subcase22}
\end{eqnarray}   
In the subcase (\ref{subcase21}), we use the triangle inequality to get
\begin{equation}
 |n|=|m+n_1+n_2+\cdots +n_k| \leq C|n_2|
\end{equation}
Hence, we have
\begin{equation}
\langle n \rangle^s |m|\leq C\langle n_2\rangle^s \langle m\rangle^{\frac12}\langle n_1\rangle^\frac12\,. 
\end{equation}
Thus this subcase can be treated exactly the same as the case (\ref{case1}). We omit the details. \\

For the subcase (\ref{subcase22}), observe that
\begin{equation}\label{arith}
 n^5-(m^5+n_1^5+\cdots + n_k^5) 
=  (m+n_1)^5 -m^5-n_1^5 + B\,,
\end{equation}
where $ B$ is given by
\begin{equation}\label{defofB}
 5(m+n_1)^4b + 10(m+n_1)^3b^2+10(m+n_1)^2b^3 +5(m+n_1)b^4+b^5\,.
\end{equation}
Here $b= n_2+\cdots+n_k$. Clearly we can estimate $B$ by
\begin{equation}\label{estB}
 |B|\leq 100k(m+n_1)^4|n_2|\,.
\end{equation}
On the other hand, notice that
\begin{equation}
(m+n_1)^5 -m^5-n_1^5= 5(m+n_1)m n_1(m^2+ n_1^2+m n_1)\,. 
\end{equation}
This implies
\begin{equation}\label{lB}
\begin{aligned}
 & \,\left|(m+n_1)^5 -m^5-n_1^5\right|\\
 \geq &\, \frac{15}{4}|m+n_1||m||n_1|\max\{|m|, |n_1|\}^2 \\
 \geq & \, {90k^2}(m+n_1)^4 |n_2|  \,. 
\end{aligned}
\end{equation}
From (\ref{estB}) and (\ref{lB}), we get
\begin{equation}\label{arith1}
\left| n^5-(m^5+n_1^5+\cdots + n_k^5)\right|\geq C |m||n_1|^2 \langle n_2\rangle
 \geq C|m|^3\,.
\end{equation}

Henceforth, at least one of following statements must hold:
\begin{eqnarray}
 &  \left| \lambda + n^5\right|\geq C |m|^3 \,, &  \label{I} \\
&   \left |(\lambda-\lambda_1-\cdots -\lambda_k) +m^5\right|\geq C|m|^3\,, &  \label{II}\\
&  \exists  i \in \{1, \cdots, k\}\,\,\, {\text {such that}}\,\,\, |\lambda_i + n_i^5|\geq  C|m|^3\,.& \label{III}
\end{eqnarray}

For (\ref{I}), (\ref{ws1}) can be bounded by
\begin{equation}\label{I-est}
\sum_{m+n_1+\cdots+n_k=n}\int \langle n_1 \rangle^s |\widehat{u}(m,\lambda-\lambda_1-\cdots-\lambda_k)|
 |\widehat{u}(n_1,\lambda_1)|\cdots|\widehat{u}(n_k,\lambda_k)||A_{n,\lambda}|
d\lambda_1\cdots d\lambda_k d\lambda.
\end{equation}
Let $F_1$ be defined by
\begin{equation}\label{defofF1}
F_1(x,t) = \sum_n \int |A_{n, \lambda}|e^{i\lambda t}  e^{inx} d\lambda\,.
\end{equation}
Then we represent (\ref{I-est}) as
\begin{equation}\label{I-est2}
\sum_{m+n_1+\cdots+n_k=n}\int  \wh F_1(n, \lambda) \wh U(m, \lambda-\lambda_1-\cdots-\lambda_k)
\wh H(n_1, \lambda_1)\prod_{j=2}^k\wh U(n_j, \lambda_j) d\lambda_1\cdots d\lambda_k d\lambda\,.
\end{equation}
Here $H$ and $U$ are functions defined in (\ref{defofH}) and (\ref{defofU}) respectively.
Clearly (\ref{I-est2}) equals 
\begin{equation}\label{I-est3}
\int_{\mathbb T\times \mathbb R} F_1(x, t) H(x,t) U(x, t)^{k} dx dt\,.
\end{equation}
Utilizing H\"older inequality, we estimate it further by
\begin{equation}
 \|F_1\|_2 \|H\|_4\|U\|_{4k}^k \leq C\delta^\theta \|u\|_{Y_s}^{k+1}\,.
\end{equation}
This yields the desired estimate for the subcase (\ref{I}).\\

For the subcase of (\ref{II}), (\ref{ws1}) is estimated by
\begin{eqnarray*}
&  &\sum_{m+n_1+\cdots+n_k=n}\int\frac{\langle n_1\rangle^s |A_{n, \lambda}|}{\langle \lambda+n^5 \rangle^\frac{1}{2}}\ \langle (\lambda-\lambda_1-\cdots-\lambda_k)+m^5\rangle^{\frac12} |\widehat{u}(m,\lambda-\lambda_1-\cdots-\lambda_k)|
 \\
&  &\,\,\,\,\,\,\,\,\,\hspace{2cm} \cdot |\widehat{u}(n_1,\lambda_1)|\cdots|\widehat{u}(n_k,\lambda_k)|
d\lambda_1\cdots d\lambda_k d\lambda\,, 
\end{eqnarray*}
which is equal to
\begin{equation}\label{II-est}
 \int_{\mathbb T\times \mathbb R} F(x,t) G(x,t) H(x,t) U^{k-1}(x,t) dx dt\,.
\end{equation}
Apply H\"older inequality to control (\ref{II-est}) by
\begin{equation}\label{II-est1}
\|F\|_4\|G\|_4\|H\|_4 \|U\|_{4(k-1)}^{k-1} \leq C\delta^\theta \|u\|_{Y_s}^{k+1}\,.
\end{equation}
This completes the estimate for the subcase (\ref{II}).\\

For the contribution of (\ref{III}), we only consider $|\lambda_2+n_2^5|\geq C|m|^3$ without 
loss of generality for $i\in\{2, \cdots, k\}$.  
This is because the $|\lambda_1+n^5|\geq C|m|^3$ case can be handled similarly as (\ref{II}). 
Hence, in this case, (\ref{ws1}) can be bounded by
$$
\sum_{m+n_1+\cdots+n_k=n}\int\frac{\langle n_1\rangle^s|A_{n,\lambda}|}{\langle \lambda+n^5 \rangle^\frac{1}{2}}\ \langle \lambda_2+n_2^5\rangle^\frac{1}{2} |\widehat{u}(m,\lambda-\lambda_1-\cdots-\lambda_k)|
\prod_{j=1}^k|\widehat{u}(n_j,\lambda_j)|
d\lambda_1\cdots d\lambda_k d\lambda.
$$
Now set a function $I$ by
\begin{equation}\label{defofI}
I(x,t) = \sum_n\int \langle \lambda+n^5\rangle^\frac{1}{2}|\wh u(n, \lambda)|e^{i\lambda t} e^{inx} d\lambda\,.
\end{equation}
Then we estimate (\ref{ws1}) by
\begin{equation}\label{est-III1}
\int_{\mathbb T\times \mathbb R} F(x,t)H(x,t)I(x,t)U^{k-1}(x,t)dxdt\,,
\end{equation}
which is majorized by
\begin{equation}
 \|F\|_4\|H\|_4\|I\|_2\|U\|_\infty^{k-1}\,.
\end{equation}

Notice this time we cannot simply use H\"older's inequality to get $\delta$ as we did before because there is no way of making any above 4 or 2 even a little bit smaller. But this can be fixed as follows.

First observe that  
$$
\|u\|_{0, 0}\leq \delta^{1/2}\|u\|_{L^2_xL^\infty_t}\leq C\delta^{1/2}\|u\|_{0, \frac12+}\,,
$$
for $u$ is supported in a $\delta$-sized interval in time variable. Thus by interpolation,
we get
\begin{equation}\label{gain}
\|u\|_{0, \frac{3}{10}}\leq C\delta^{\frac15-}\|u\|_{0, \frac12}\,. 
\end{equation}
Since $U$ can be assumed to be a function supported in a $\delta$-sized time interval, 
we may put the same assumption to $H$. 
Henceforth, we have
\begin{equation}\label{III-est2}
\|H\|_4\leq C\|H\|_{0, \frac{3}{10}}\leq C\delta^{\frac15-}\|H\|_{0, \frac12}\leq C\delta^{\frac15-}\|u\|_{Y_s}\,. 
\end{equation}
Also note that
\begin{equation}\label{III-est3}
\|I\|_2\leq \|u\|_{0,\frac12} \leq \|u\|_{Y_s}\,.
\end{equation}
and
\begin{equation}\label{III-est4}
\|U\|_{\infty}\leq C\|u\|_{Y_s}\,. 
\end{equation}
From (\ref{III-est2}), (\ref{III-est3}) and (\ref{III-est4}), we can estimate (\ref{ws1}) by
$ C\delta^{\frac15-}\|u\|_{Y_s}^{k+1} $ as desired. Therefore we finish our discussion for 
the case (\ref{case2}).

\subsection{Case (\ref{case3})}  In this case, let us further consider two subcases.
\begin{eqnarray}
  &  |m|^4 \leq 1000k^2|n_2|^4 |n_3|    &   \label{subcase31}\\
  &  |m|^4 > 1000k^2|n_2|^4 |n_3|   &  \label{subcase32}
\end{eqnarray}   

For the contribution of (\ref{subcase31}), we observe that from (\ref{subcase31}), 
$$
 |m|\leq C|n_1|^{1/2}|n_2|^{1/2}|n_3|^{1/4}\,,
$$
since $|n_2|\leq |n_1|$. 
This implies immediately 
\begin{equation}\label{shift31}
 \langle n\rangle^s |m|\leq C|m|^{s+1}\leq 
\langle m\rangle^s \langle n_1\rangle^{1/2} \langle n_2\rangle^{1/2} \langle n_3\rangle^{1/4}.
\end{equation}
Introduce a new function $H_1$ defined by
\begin{equation}\label{defofH1}
 H_1(x, t) =\sum_n \int_{\mathbb R} \langle n\rangle^{1/4}|\wh u(n, \lambda)| e^{i\lambda t} e^{inx} d\lambda\,.
\end{equation}
As before, in this case, we bound (\ref{ws1}) by
\begin{equation}\label{estCase31}
\int_{\mathbb T\times \mathbb R} F(x,t)H(x,t)G^2(x,t)H_1(x,t)U^{k-3}(x,t) dx dt\,.
\end{equation}
Then H\"older inequality yields 
\begin{equation}
(\ref{ws1})\leq C\delta^\theta\|F\|_4\|H\|_{4}\|G\|_{6+}^2 \|H_1\|_{6+}\|U\|_{\infty}^{k-3}\,.
\end{equation}
Clearly
$$\|H_1\|_{6+}\leq  C\|u\|_{\frac14+, \frac12}\leq C\|u\|_{Y_s} \,.$$ 
Hence we obtain the desired estimate for the subcase (\ref{subcase31}).\\

We now turn to the contribution of (\ref{subcase32}). Clearly we have
\begin{equation}\label{arith2}
\begin{aligned}
 &\, n^5-(m^5+ n_1^5+ \cdots +n_k^5) \\
 =& \, 
 5m^4(n_1+b)+10m^3(n_1+b)^2+10m^2(n_1+b)^3+5m(n_1+b)^4 \\
&\,\,\,
 + 5(n_1+b)n_1 b(n_1^2+b^2+n_1b)+O(n_2^4n_3) 
\,,
\end{aligned}
\end{equation}
since $|n_2|\geq |n_3|\geq\cdots\geq |n_k|$. 
From  (\ref{case3}), (\ref{subcase32}), (\ref{arith2}) and $n_1+b\neq 0$, we have
\begin{equation}\label{arith3}
 \left|n^5 -\left(m^5+n_1^5+\cdots+n_k^5\right)\right| \geq  C|m|^4\,.
\end{equation} 
This is similar to (\ref{arith1}). Hence again we reduce the problems to (\ref{I}), (\ref{II}), and (\ref{III}), which are all done in Subsection \ref{Hicase2}. Therefore we finish the case of (\ref{case3}).

Putting all cases together, we obation
\begin{equation}\label{ws1est}
\|w\|_{s, -\frac12}\leq C\delta^\theta \|u\|_{Y_s}^{k+1}\,.
\end{equation}

The desired estimate 
\begin{equation}\label{extra}
\left(\sum_n\langle n\rangle^{2s}\left(\int\frac{|\widehat{w}(n,\lambda)|}{\langle\lambda +n^5\rangle}d\lambda\right)^2\right)^\frac{1}{2}\leq 
 C\delta^\theta \|u\|_{Y_s}^{k+1}
\end{equation}
can be obtained similarly, and we omit the details. Therefore we complete the proof of Proposition {\ref{propofw}} by combining (\ref{ws1est}) and (\ref{extra}). \\

\section{Local well-posedness of (\ref{gKdV2})}\label{LWP2}
\setcounter{equation}0

We now start to derive the local well-posedness of (\ref{gKdV2}).  Without loss of generality, we only need to 
consider the well-posedness of the Cauchy problem:
\begin{equation}\label{KdV-3}
\begin{cases}
u_t+ \p_x^5u+ u^k u_x u_x=0\\
u(x,0)=\phi(x),\qquad x\in\mathbb{T},\ t\in\mathbb{R}
\end{cases}.
\end{equation}


Now let $w$ be the nonlinear function defined by
\begin{equation}\label{defofw-1}
w = u^k u_x u_x\,.
\end{equation}


As before, we need the following estimate 
on the nonlinear function $w$, in order to establish a contraction on the space $\{u: \|u\|_{Y_s}\leq M\}$ for some 
$M>0$. A proof of Proposition \ref{propofw-1} will appear in Section \ref{proofofw-1}.

\begin{proposition}\label{propofw-1}
For $s >1$, there exists $\theta>0$ such that, for the nonlinear function $w$ 
 given by (\ref{defofw-1}),  
\begin{equation}
\|w\|_{s, -\frac12} + \left(\sum_n\langle n\rangle^{2s}\left(\int\frac{|\widehat{w}(n,\lambda)|}{\langle\lambda+n^5\rangle}d\lambda\right)^2\right)^\frac{1}{2} \leq C\delta^\theta \|u\|_{Y_s}^{k+2}.
\end{equation}
Here $C$ is a constant independent of $\delta$ and $u$. 
\end{proposition}

By applying Duhamel principle, we reduce the problem to the well-posedness of corresponding integral equation associated to (\ref{KdV-3}) 
\begin{equation}
u(x,t) = e^{-t\partial_x^5}\phi(x)-\int_0^te^{-(t-\tau)\partial_x^5}w(x,\tau)d\tau,
\end{equation}
where $w$ is defined as in (\ref{defofw-1}). Using the local smooth truncation, we only need
 to seek a local solution of
$$u(x,t) = \psi_\delta(t)e^{-t\partial_x^5}\phi(x)-\psi_\delta(t)\int_0^te^{-(t-\tau)\partial_x^5}w(x,\tau)d\tau.$$
Let $T_1$ be an operator given by
\begin{equation}\label{defofT1}
T_1u(x,t):= \psi_\delta(t)e^{-t\partial_x^5}\phi(x)-\psi_\delta(t)\int_0^te^{-(t-\tau)\partial_x^5}w(x,\tau)d\tau.
\end{equation}

\begin{proposition}\label{propT1u}
Let $s\geq 1$ and $T_1 $ be the operator defined as in (\ref{defofT1}).  Then there exits a
positive number $\theta$ such that
\begin{equation}\label{estofT1u}
\|T_1u\|_{Y_s}\leq C\left(\|\phi\|_{H^s} + \delta^\theta \|u\|_{Y_s}^{k+2}\right)\,. 
\end{equation}
Here $C$ is a constant independent of $\delta$. 
\end{proposition}

\begin{proof}
Since $Tu=\mathcal L u +\mathcal N u$, Proposition \ref{propT1u} follows from
Lemma {\ref{estLu}}, Lemma {\ref{estNu}} and Proposition {\ref{propofw-1}}.   
\end{proof}

Proposition \ref{propT1u} yields that for $\delta$ sufficiently small, $T$ maps a ball in $Y_s$ 
into itself.  By a similar argument as in the proof of Proposition {\ref{propofw-1}},  
one obtains, for $s\geq 1$, 
\begin{equation}\label{contra1}
\|T_1u-T_1v\|_{Y_s}\leq \delta^\theta C(\|u\|_{Y_s}, \|v\|_{Y_s}) \|u-v\|_{Y_s}\,.
\end{equation}
Here $ C(\|u\|_{Y_s}, \|v\|_{Y_s})$ is a real number depending on $\|u\|_{Y_s} $ and $\|v\|_{Y_s}$.
Henceforth, for $\delta>0$ sufficiently small, $T_1$ is a contraction and the local well-posedness 
follows from Picard's fixed-point theorem.  This completes the proof of Theorem {\ref{thm-gKdV3}}.\\

\section{Proof of Proposition {\ref{propofw-1}}}\label{proofofw-1}
\setcounter{equation}0

We only present details for proving
\begin{equation}\label{estW-1}
\|w\|_{s, -\frac12} \leq C\delta^\theta \|u\|_{Y_s}^{k+2}\,.
\end{equation} 
The estimates for the extra term 
$$
 \left(\sum_n\langle n\rangle^{2s}\left(\int\frac{|\widehat{w}(n,\lambda)|}{\langle\lambda+n^5\rangle}d\lambda\right)^2\right)^\frac{1}{2} \leq C\delta^\theta \|u\|_{Y_s}^{k+2}\,
$$
are similar, and we omit the details. For simplicity, we assume $\delta=1$.

From the definition of $w$ in (\ref{defofw-1}), we may write $\wh w(n, \lambda)$ as
\begin{equation}\label{w-1}
\sum_{m_1+m_2+n_1+\cdots+n_k=n}m_1m_2\int\widehat{u}(m_1,\lambda-\mu-\lambda_1-\cdots-\lambda_k)
\widehat{u}(m_2,\mu)\prod_{j=1}^k\widehat{u}(n_j,\lambda_j) d\mu d\lambda_1\cdots d\lambda_k.
\end{equation}

Notice that  
\begin{equation}
 w =\frac{1}{k+1}\p_x u^{k+1} \p_x u\,. 
\end{equation}
Hence for free we can put additional conditions for $m_1, m_2, n_1, \cdots, n_k$:
\begin{equation}\label{cond1-1}
m_1+ n_1+\cdots+n_k\neq 0
\end{equation}
and 
\begin{equation}\label{cond1-2}
m_2+ n_1+\cdots+n_k\neq 0\,.
\end{equation}
Without loss of generality, we assume that
\begin{equation}
 |m_1|\geq |m_2|\,\,\,{\text{and}}\,\,\, |n_1|\geq \cdots\geq |n_k|\,. 
\end{equation}
Henceforth, by duality, $\|w\|_{s, -\frac12}$ is bounded by
\begin{equation}\label{w1s1}
\begin{aligned}
 &\sum_{\substack{m_1+m_2+n_1+\cdots+n_k=n\\m_2+n_1+\cdots+n_k\neq0\\
   |m_1|\geq |m_2|\\
 |n_1|\geq\cdots\geq |n_k|}}\!\int\frac{\langle n\rangle^s |A_{n,\lambda}||m_1||m_2|}{\langle \lambda +n^5 \rangle^\frac{1}{2}}|\widehat{u}(m_1,\lambda-\mu-\lambda_1-\cdots-\lambda_k)| |\wh u(m_2, \mu)|  \\
&\hspace{3cm}\cdot
 |\widehat{u}(n_1,\lambda_1)|\cdots|\widehat{u}(n_k,\lambda_k)|
d\mu d\lambda_1\cdots d\lambda_k d\lambda.
\end{aligned}
\end{equation}
Here the  sequence 
 $\{A_{n, \lambda}\}$ satisfying 
\begin{equation}\label{AnL1}
\sum_{n\in\mathbb Z} \int_{\mathbb R}|A_{n, \lambda}|^2 d\lambda \leq 1\,,
\end{equation}

Carrying  on the similar idea as before, we want to either distribute 
$m_1, m_2$ into some $\wh u$'s or get some decay factor to cancel $m_1$.
More precisely, let us consider two cases. 
\begin{eqnarray}
  &   |m_1+m_2| \leq 1000k^2|n_1|\,; &  \label{case1-1}\\
  &   |m_1+m_2| >  1000k^2|n_1| \,; &\label{case2-1}
\end{eqnarray}

\subsection{Case (\ref{case1-1})}
In this subcase, we have
\begin{equation}\label{shift-1}
 |n|^s\leq C|n_1|^s
\end{equation}
since $n=m_1+m_2+n_1+\cdots + n_k$ and $|n_1|\geq\cdots\geq |n_k|$. 
Hence we may distribute $n^s$ into $\wh u(n_1, \lambda_1)$ so that
(\ref{w1s1}) is estimated by
\begin{equation}\label{est-case1-1-1}
  \int  |F(x,t)| |G_1(x,t)|^2 |H(x,t)| |U(x,t)|^{k-1}        dx dt \,,
\end{equation}
where $F, H,$ and $ U$ are functions defined as in (\ref{defofF}), (\ref{defofH}) and (\ref{defofU}), respectively, and $G_1$ is given by
\begin{equation}\label{defofG1}
 G_1(x,t) = \sum_n \int |n||\wh u(n, \lambda)| e^{i\lambda t} e^{inx} d\lambda\,. 
\end{equation}
By a use of H\"older inequality and $s\geq 1$, we dominate (\ref{est-case1-1-1}) by
\begin{equation}\label{est-case1-1-2}
 \|F\|_4 \|G_1\|_4^2 \|H\|_4 \|U\|_\infty^{k-1} \,,
\end{equation}
which is clearly bounded by
$
 C\|u\|_{Y_s}^{k+2}\,,
$
as desired. 

\subsection{Case (\ref{case2-1})}
In this case, we have
\begin{equation}\label{largem1}
 |m_1|\geq 500k^2|n_1|\,. 
\end{equation}

First we consider the subcase $|m_2|\leq |n_1|$. In this subcase, 
 we get 
\begin{equation}\label{arith-1-2}
\begin{aligned}
 &\, n^5-(m_1^5 + m_2^5+ n_1^5 + \cdots +n_k^5) \\
 =& \, 
 5m_1^4(n_1+b_2)+10m_1^3(n_1+b_2)^2+10m_1^2(n_1+b_2)^3+5m_1(n_1+b_2)^4 \\
&\,\,\,
 + 5(n_1+b_2)n_1 b_2(n_1^2+ b_2^2+ n_1 b_2)+O(m_2^4n_2) + O(n_2^4m_2) +
 O(n_2^4n_3)\,,
\end{aligned}
\end{equation}
where $b_2= m_2+ n_2+\cdots +n_k$.
Since $m_2+n_1+\cdots +n_k\neq 0$, $ n_1+b_2\neq 0$. Notice that in this
case $|n_1|\ll |m_1|$. Then we  have either
\begin{eqnarray}
 & \max\{ m_2^4|n_2|, n_2^4|m_2|, n_2^4|n_3|\} \geq  \frac{1}{100}m_1^4  &\label{subcase2-1-1}  \\
{\rm or}  & |n^5-(m_1^5 + m_2^5+ n_1^5\cdots +n_k^5)|\geq m_1^4 \,.   &   \label{subcase2-1-2} 
\end{eqnarray}
(\ref{subcase2-1-1}) implies 
\begin{equation}\label{shift-2}
 \langle n\rangle^s |m_1||m_2|\leq 
 C|m_1|^s|m_1||m_2|\leq  
C\max\{|n_1|^s|n_2|^{s/4}|m_1||m_2|,\, |n_1|^{s/4}|n_2|^{s} |m_1||m_2| \}\,. 
\end{equation}
Henceforth we estimate (\ref{w1s1}) by
\begin{equation}\label{est-case1-1-3}
  \int  |F(x,t)| |G_1(x,t)|^2 |H(x,t)| 
 |H_2(x,t)|  |U(x,t)|^{k-2}        dx dt \,,
\end{equation}
where $H_2$ is defined by
\begin{equation}\label{defofH2}
 H_2(x,t) = \sum_n \int \langle n\rangle^{s/4}|\wh u(n, \lambda)| e^{i\lambda t} e^{inx} d\lambda\,. 
\end{equation}
Using H\"older inequality, we have 
\begin{equation}\label{est-case1-1-4}
 \|F\|_4\|G_1\|_6^2\|H\|_4\|H_2\|_6\|U\|_\infty^{k-2}\leq 
 C \|u\|_{Y_s}^{k+2}\,,
\end{equation}
since $s>1$. This finishes the case of (\ref{subcase2-1-1}). 
If (\ref{subcase2-1-2}) holds, then one of the following statements must be true:
\begin{eqnarray}
 & |\lambda +n^5|\geq m_1^4 & \label{I-1}  \\
 &  |\lambda-\mu-\lambda_1-\cdots -\lambda_k+m_1^5 |\geq m_1^4 &   \label{II-1} \\
 & |\mu +m_2^5|\geq m_1^4   &    \label{III-1}\\
 & \exists j\in\{1, \cdots, k\}\,, \,\,\, |\lambda_j+n_j^5|\geq m_1^4  &    \label{IV-1}
\end{eqnarray}
The cases (\ref{I-1}), (\ref{II-1}), (\ref{III-1}) and (\ref{IV-1}) can be done 
similarly as the cases (\ref{I}), (\ref{II}) and (\ref{III}). We omit the details. 
This completes the discussion on the subcase $|m_2|\leq |n_1|$.\\

We now turn to the subcase $|m_2|>|n_1|$. In this subcase, observe that 
\begin{equation}\label{arith-1-1}
\begin{aligned}
 &\, n^5-(m_1^5 + m_2^5+ n_1^5+\cdots +n_k^5) \\
 =& \, 
 5m_1^4(m_2+b_1)+10m_1^3(m_2+b_1)^2+10m_1^2(m_2+b_1)^3+5m_1(m_2+b_1)^4 \\
&\,\,\,
 + 5(m_2+b_1)m_2 b_1(m_2^2+ b_1^2+m_2b_1)+O(n_1^4n_2) \\
 = &  5(m_2+b_1)m_1(m_1+m_2+b_1)\left(m_1^2+(m_2+b_1)^2+m_1(m_2+b_1)\right)\\
   &\, \,\,\, +  5(m_2+b_1)m_2 b_1(m_2^2+ b_1^2+m_2b_1)+O(n_1^4n_2)                                                \,,
\end{aligned}
\end{equation}
where $b_1 =n_1 +\cdots + n_k$.  . Notice that, 
from (\ref{case2-1}), 
\begin{equation}\label{ll1}
\begin{aligned} 
  & \left| 5(m_2+b_1)m_1(m_1+m_2+b_1)\left(m_1^2+(m_2+b_1)^2+m_1(m_2+b_1)\right)\right| \\
\geq & \, 2000k^2|m_2+b_1||m_1|^3\langle n_1\rangle\,.
\end{aligned}
\end{equation}
Clearly we also have
\begin{equation}\label{ll2}
|5(m_2+b_1)m_2 b_1(m_2^2+ b_1^2+m_2b_1)|\leq 15k|m_2+b_1||m_1|^3|n_1|\,.
\end{equation}
Since $ m_2+b_1\neq 0$, we reduce the problem to either
 \begin{eqnarray}
 &  n_1^4|n_2| \geq  \frac{1}{100}|m_1|^3 \langle n_1 \rangle  &\label{subcase2-2-1}  \\
{\rm or}  & |n^5-(m_1^5 + m_2^5+ n_1^5\cdots +n_k^5)|\geq m_1^3 \,.   &   \label{subcase2-2-2} 
\end{eqnarray}
Notice that from (\ref{subcase2-2-1}), we obtain
\begin{equation}\label{shift221}
 \langle n\rangle^s |m_1||m_2|\leq C|m_1|^{s+1}|m_2|\leq |m_1||m_2||n_1|^s|n_2|^{s/3}\,. 
\end{equation}
Then the desired estimate follows by using the same method as in (\ref{shift-2}).
The case (\ref{subcase2-2-2}) can be handled exactly the same as the case (\ref{subcase2-1-2}).
Hence we complete the proof for the subcase $|m_2|>|n_1|$. Therefore 
the discussion on Case (\ref{case2-1}) is done.\\

\end{document}